\documentclass[12pt,british,a4paper,reqno]{amsart}

	\usepackage[T1]{fontenc}
	\usepackage[latin9]{inputenc}
	\usepackage[british]{babel}

	\usepackage{amsmath}
	\usepackage{amssymb}
	\usepackage{amsthm}
	\usepackage{bbm}
	\usepackage{braket}
	\usepackage{xcolor}
	\usepackage[shortlabels]{enumitem}
	\usepackage{fancyhdr}
	\usepackage{graphicx}
	\usepackage{ifsym}
	\usepackage{indentfirst}
	\usepackage{lmodern}
	\usepackage{mathrsfs}
	\usepackage{mathtools}
	\usepackage{oplotsymbl} 
	\usepackage{proof}
	\usepackage{qtree}
	\usepackage{scalerel}
	\usepackage{setspace}
	\usepackage{stmaryrd}
	\usepackage{tikz}
	\usepackage{tikz-cd}
	\usepackage{wasysym}
	\usepackage{xparse}
            \makeatletter
            \def\amsbb{\use@mathgroup \M@U \symAMSb}
            \makeatother
        \usepackage{bbold}

	\usepackage[hyperfootnotes=false,hidelinks]{hyperref}
		\usepackage[capitalise]{cleveref}
		\usepackage{bookmark}
	\usepackage{geometry}

	\bookmarksetup{numbered,open}
    \setlist{listparindent=\parindent} 

	\renewcommand{\geq}{\geqslant}
	\renewcommand{\leq}{\leqslant}
	\renewcommand{\phi}{\varphi}
	
	\renewcommand{\sp}{\mathrm{sp}}

	\providecommand{\corollaryname}{Corollary}
	\providecommand{\definitionname}{Definition}
        \providecommand{\definitionpropname}{Definition-Proposition}
	\providecommand{\examplename}{Example}
	\providecommand{\lemmaname}{Lemma}
	\providecommand{\notationname}{Notation}
	\providecommand{\propositionname}{Proposition}
	\providecommand{\remarkname}{Remark}
	\providecommand{\theoremname}{Theorem}
	\providecommand{\setupname}{Setup}
	\providecommand{\conjecturename}{Conjecture}
	\providecommand{\questionname}{Question}
	\providecommand{\claimname}{Claim}

	\theoremstyle{plain}
		\newtheorem{thm}{\protect\theoremname}[section] 
            \newtheorem*{thm*}{\protect\theoremname}
		
		\newtheorem{prop}[thm]{\protect\propositionname}
		\newtheorem{lem}[thm]{\protect\lemmaname}

	\theoremstyle{definition}
		\newtheorem{defn}[thm]{\protect\definitionname}

		\newtheorem{example}[thm]{\protect\examplename}
		\newtheorem{setup}[thm]{\protect\setupname}

	\theoremstyle{remark}
		\newtheorem{rem}[thm]{\protect\remarkname}
		
	\numberwithin{figure}{section}
	\numberwithin{equation}{section}

	\makeatletter
	
		\newcommand\ackname{Acknowledgements}
		\newenvironment{acknowledgements}{%
			\medskip
			\bgroup
			\list{}{\labelwidth\z@
			\leftmargin3pc \rightmargin\leftmargin
			\listparindent\normalparindent \itemindent\z@
			\parsep\z@ \@plus\p@
			
				}%
				\Small
				\item[\hskip\labelsep\scshape\ackname.]%
			}{%
			\endlist\egroup
			}
			
	\makeatother

\let\amph=& 
	\usetikzlibrary{matrix,arrows,decorations.pathmorphing,positioning,decorations.pathreplacing}
	\tikzset{commutative diagrams/.cd, 
		mysymbol/.style = {start anchor=center, end anchor = center, draw = none}}
	\tikzcdset{every label/.append style = {font = \footnotesize}}
	\tikzcdset{scale cd/.style={every label/.append style={scale=#1},
    cells={nodes={scale=#1}}}}

	\newcommand{\BE}{\amsbb{E}}

	\newcommand{\BQ}{\amsbb{Q}}

    \newcommand{\cat}[1]{\mathcal{#1}}
    \newcommand{\fun}[1]{\mathsf{#1}}
    \newcommand{\adj}[1]{\mathscr{#1}}
    \newcommand{\yoneda}{\amsbb{Y}}

		\newcommand{\Ab}{\operatorname{\mathrm{Ab}}\nolimits}

		\newcommand{\op}{\mathrm{op}}

		\newcommand{\sse}{\subseteq}

            \newcommand{\Iso}[1]{\operatorname{\mathrm{Iso}}\nolimits{#1}}

		\newcommand{\Ker}{\operatorname{Ker}\nolimits}
		\newcommand{\Cok}{\operatorname{Coker}\nolimits}

		\newcommand{\rad}{\operatorname{rad}\nolimits}

		\newcommand{\onto}{\rightarrow\mathrel{\mkern-14mu}\rightarrow}

		\newcommand{\id}[1]{\mathrm{id}_{#1}}

		\newcommand{\indxx}[1]{\operatorname{\mathrm{index}}_{#1}\nolimits}
  \newcommand{\Ind}[1]{\operatorname{\mathrm{Ind}}_{#1}\nolimits}

		\NewDocumentCommand{\newindx}{O{} m}{\operatorname{\mathrm{ind}}_{#2}^{#1}\nolimits}

		\newcommand{\fs}{\mathfrak{s}}
		
		\newcommand{\rMod}[1]{\operatorname{\mathrm{Mod}}\nolimits{#1}}
            \newcommand{\rProj}[1]{\operatorname{\mathrm{Proj}}\nolimits{#1}}

            \newcommand{\rmod}[1]{\operatorname{\mathrm{mod}}\nolimits{#1}}
            \newcommand{\rproj}[1]{\operatorname{\mathrm{proj}}\nolimits{#1}}
            \newcommand{\lmod}[1]{{#1}\operatorname{\mathrm{mod}}\nolimits}
            \newcommand{\lproj}[1]{{#1}\operatorname{\mathrm{proj}}\nolimits}

            \newcommand{\add}[1]{\operatorname{\mathrm{add}}\nolimits{#1}}

	\newcommand{\lgldim}[1]{\operatorname{l.\!gl.\!dim}\nolimits{#1}}
	\newcommand{\rgldim}[1]{\operatorname{r.\!gl.\!dim}\nolimits{#1}}

	\newcommand{\pd}[2]{\operatorname{pd}_{#1}#2}
		 
	\newcommand{\rpd}[2]{\operatorname{fpd}_{#1}#2}
	\newcommand{\lpd}[2]{\operatorname{fpd}_{#1}#2}

	\newcommand{\deff}{\coloneqq}
	
	\newcommand\restr[2]{{\left.\kern-\nulldelimiterspace#1
						\right|_{#2}}}

    \makeatletter
    
    \renewcommand{\andify}{%
		\nxandlist{\unskip, }{\unskip{} \@@and~}{\unskip \penalty-2 \space \@@and~}} 
    \renewcommand\author@andify{%
  		\nxandlist {\unskip ,\penalty-1 \space\ignorespaces}%
		{\unskip {} \@@and~}%
		{\unskip \penalty-2 \space \@@and~}}
    
    \makeatother

    \let\oldtocsection=\tocsection
    \let\oldtocsubsection=\tocsubsection
    \renewcommand{\tocsection}[2]{\hspace{0em}\oldtocsection{#1}{#2}}
    \renewcommand{\tocsubsection}[2]{\hspace{2em}\oldtocsubsection{#1}{#2}}

\begin{document}

\title{The index in \texorpdfstring{\lowercase{$d$}}{d}-exact  categories}
\author[Fedele]{Francesca Fedele}
\address{
		School of Mathematics\\
		University of Leeds\\
		Leeds LS2 9JT\\
		UK
	}
    \email{f.fedele@leeds.ac.uk}
\author[J{\o{}}rgensen]{Peter J{\o{}}rgensen}
	\address{
		Department of Mathematics\\
		Aarhus University\\
		8000 Aarhus\\
		Denmark
	}
    \email{peter.jorgensen@math.au.dk}

\author[Shah]{Amit Shah}
    \email{a.shah1728@gmail.com}

\date{\today}



\begin{abstract}
Starting from its original definition in module categories with respect to projective modules, the index has played an important role in various aspects of homological algebra, categorification of cluster algebras and $K$-theory. In the last few years, the notion of index has been generalised to several different contexts in (higher) homological algebra, typically with respect to a (higher) cluster-tilting subcategory $\mathcal{X}$ of the relevant ambient category $\mathcal{C}$. The recent tools of extriangulated and higher-exangulated categories have permitted some conditions on the subcategory $\mathcal{X}$ to be relaxed. In this paper, we introduce the index with respect to a generating, contravariantly finite subcategory of a $d$-exact category that has $d$-kernels. We show that our index has the important property of being additive on $d$-exact sequences up to an error term.
\end{abstract}

\maketitle

\setcounter{tocdepth}{1}

\section{Introduction}

Auslander and Reiten first introduced the concept of an index of a module in \cite[Sec.~3]{AuslanderReiten-Modules-determined-by-their-composition-factors}, defined as $[P_0]-[P_1]$ in a suitable Grothendieck group when $P_1 \xrightarrow{} P_0 \xrightarrow{} M \xrightarrow{} 0$ is a minimal projective presentation of a finitely generated module over a finite-dimensional algebra. 

Starting from the above, the idea of an index has then been generalised to many different contexts.
Palu introduced the index with respect to a cluster-tilting subcategory of a triangulated category in \cite[Sec.~2.1]{Palu-Cluster-characters-for-2-Calabi-Yau-triangulated-categories}. Padrol--Palu--Pilaud--Plamondon then showed in \cite[Prop~4.11]{PadrolPaluPilaudPlamondon-Associahedra-for-finite-type-cluster-algebras-and-minimal-relations-between-g-vectors} that this can be recovered using the theory of extriangulated categories, leading to the definition of the index in a triangulated category $\mathcal{C}$ with respect to a contravariantly finite, 
rigid subcategory $\mathcal{X}$
in \cite{JorgensenShah-index}. For $C\in\mathcal{C}$, the index is defined as the class $[C]_\mathcal{X}$ in the Grothendieck group $K_0(\cat{C},\BE_{\cat{X}},\fs_{\cat{X}})$, where $(\cat{C},\BE_{\cat{X}},\fs_{\cat{X}})$ is an extriangulated structure on $\cat{C}$ relative to the triangulated structure.
In \cite{fedele2024index}, taking inspiration from methods used in \cite{CGMZ} by Conde--Gorksy--Marks--Zvonareva, we widened the theory by showing that the assumption on $\mathcal{X}$ being rigid can be dropped, augmenting the class of subcategories that admit a well-defined index.

The present paper uses similar methods to continue the investigation on the index, but in the different direction of higher homological algebra. 
Let $d$ be a positive integer. 
The index with respect to a $d$-cluster tilting subcategory of a triangulated or abelian category has been introduced by J{\o}rgensen in \cite[Def.~3.3]{Jorgensen-Tropical-friezes-and-the-index-in-higher-homological-algebra} and Reid in \cite[Sec.~1]{Reid-Modules-determined-by-their-composition-factors-in-higher-homological-algebra}, respectively. 
Since we focus on the exact category setting here, let us assume $(\cat{C},\adj{E})$ is a skeletally small exact category (see \cite[Def.~2.1]{Buhler-exact-categories}) and that
$\cat{T}\sse\cat{C}$ is a $d$-cluster tilting subcategory in the sense of \cite[Def.~4.13]{Jasso-n-abelian-and-n-exact-categories}. 
Then, for $C\in\cat{C}$, by the dual of \cite[Prop.~4.15]{Jasso-n-abelian-and-n-exact-categories}, there is an $\adj{E}$-acyclic complex 
\begin{equation*}
\label{eqn:T-res-for-C}
\begin{tikzcd}
0 
	\arrow{r}
& T_{d-1}
	\arrow{r}
& T_{d-2}
	\arrow{r}
& \cdots
	\arrow{r}
& T_{0}
	\arrow{r}
& C 
	\arrow{r}
& 0
\end{tikzcd}
\end{equation*}
with $T_{i}\in\cat{T}$ for $0 \leq i \leq d-1$. 
In this case, 
the \emph{index of $C$ with respect to $\cat{T}$} is 
\[
\indxx{\cat{T}}(C) \deff \sum_{i=0}^{d-1} (-1)^{i} [T_{i}]^{\sp}
\]
viewed as an element of the split Grothendieck group $K_{0}^{\sp}(\cat{T})$ of $\cat{T}$.

As an application of \cite[Thm.~4.5]{OS23}, one can verify \cref{thm_dct_abelian} where the hypotheses of \cite[Thm.~4.5]{OS23} are satisfied using arguments analogous to those in \cite[Sec.~6]{OS23}. 
In particular, the isomorphism below suggests that one can interpret the class $[C]_{\cat{T}}$ as the index of $C$ with respect to $\cat{T}$.

\begin{thm}\label{thm_dct_abelian}
(cf.\ \cite[Thm.~6.5]{OS23})
Let $(\cat{C},\adj{E})$ be a skeletally small exact category and $\cat{T}\sse\cat{C}$ a $d$-cluster tilting subcategory. Consider the relative exact category $(\cat{C},\adj{E}_{\cat{T}})$ as obtained via \cref{prop:relative-d-exact-subcategory} (with $d=1$). 
There is an isomorphism
\begin{equation*}
\begin{aligned}[t]
K_{0}(\cat{C},\adj{E}_{\cat{T}}) 		&\overset{\cong}{\longrightarrow}		K_{0}^{\sp}(\cat{T}) \\[5pt]
[C]_{\cat{T}}					&\overset{}{\longmapsto}					\indxx{\cat{T}}(C)  \\[5pt]
[T]_{\cat{T}}					&\overset{}{\longmapsfrom}				[T]_{\cat{T}}^{\sp}.
\end{aligned}
\end{equation*} 
\end{thm}

The higher-dimensional version of an exact category is called a \emph{$d$-exact category} and was introduced by Jasso 
\cite[Def.~4.2]{Jasso-n-abelian-and-n-exact-categories} (or see \cref{def:d-exact-category}). 
Let $(\cat{C},\adj{E})$ be a skeletally small, $d$-exact category and suppose $\cat{X}\sse\cat{C}$ is a full subcategory. 
Motivated by \cref{thm_dct_abelian}, we define 
the \emph{index of $C\in\cat{C}$ with respect to $\cat{X}$} to be the class $[C]_\cat{X}$ in $K_0(\cat{C},\adj{E}_\cat{X})$, where $(\cat{C},\adj{E}_\cat{X})$ is  the relative $d$-exact category defined as in \cref{prop:relative-d-exact-subcategory} and its Grothendieck group as in \cref{def:grothendieck-group}.

One of the key properties of the classical index is that it is additive up to a well-behaved error term on triangles \cite[Prop.~2.2]{Palu-Cluster-characters-for-2-Calabi-Yau-triangulated-categories}, $(d+2)$-angles \cite[Thm.~C]{Jorgensen-Tropical-friezes-and-the-index-in-higher-homological-algebra}, and short exact sequences \cite[Thm.~C]{Reid-Modules-determined-by-their-composition-factors-in-higher-homological-algebra}.
Importantly, such additivity permits use of the index to build cluster characters \cite{Palu-Cluster-characters-for-2-Calabi-Yau-triangulated-categories} and tropical friezes \cite{Guo,Jorgensen-Tropical-friezes-and-the-index-in-higher-homological-algebra}.
Using methods similar to those in \cite{CGMZ} and \cite{fedele2024index}, in \cref{sec:main-results} we prove our main result, showing that our index is also additive on $d$-exact sequences in $\adj{E}$ up to an error term.

\begin{thm}\label{theoremA}
Suppose $(\cat{C},\adj{E})$ is a skeletally small, idempotent complete, $d$-exact category that has $d$-kernels. 
Let $\cat{X}\sse\cat{C}$ be a full, contravariantly finite, additive subcategory that is closed under direct summands and is also generating, see Definition~\ref{def:generating}. 
Then there is a group homomorphism 
$\theta_{\cat{X}}\colon K_{0}(\rmod{\cat{X}}) \to K_{0}(\cat{C},\adj{E}_{\cat{X}})$, satisfying: if 
\[
\begin{tikzcd}
A_{d+1}
    \arrow[tail]{r}{\partial^{A}_{d+1}}
&A_{d}
    \arrow{r}{\partial^{A}_{d}}
&\cdots
    \arrow{r}{\partial^{A}_{3}}
&A_{2}
    \arrow{r}{\partial^{A}_{2}}
&A_{1}
    \arrow[two heads]{r}{\partial^{A}_{1}}
&A_{0}
\end{tikzcd}
\]
is a $d$-exact sequence in $\adj{E}$, then
$\theta_{\cat{X}}([\Cok\left(\restr{\cat{C}(-,\partial^{A}_{1})}{\cat{X}}\right)])
    = \sum_{i=0}^{d+1}(-1)^{i}[A_{i}]_{\cat{X}}$.
\end{thm}

Moreover, we note that the morphism $\theta_\cat{X}$ in the above result is unique with respect to a stronger property on left $d$-exact sequences in $\cat{C}$.
\begin{prop}\label{prop_uniqueness}
    In the situation of Theorem~\ref{theoremA}, let
    \[
\begin{tikzcd}
A_{d+1}
    \arrow[tail]{r}{\partial^{A}_{d+1}}
&A_{d}
    \arrow{r}{\partial^{A}_{d}}
&\cdots
    \arrow{r}{\partial^{A}_{3}}
&A_{2}
    \arrow{r}{\partial^{A}_{2}}
&A_{1}
    \arrow{r}{\partial^{A}_{1}}
&A_{0}
\end{tikzcd}
\]
be a left $d$-exact sequence in $\cat{C}$. Then $\theta_{\cat{X}}([\Cok\left(\restr{\cat{C}(-,\partial^{A}_{1})}{\cat{X}}\right)])
    = \sum_{i=0}^{d+1}(-1)^{i}[A_{i}]_{\cat{X}}$ and $\theta_\cat{X}$ is unique with respect to this property.
\end{prop}

We observe here that \cite[Thm.~C]{Reid-Modules-determined-by-their-composition-factors-in-higher-homological-algebra} is a special case of \cref{theoremA}. Indeed, in the setup of \cite{Reid-Modules-determined-by-their-composition-factors-in-higher-homological-algebra}, one has a $d$-cluster tilting subcategory $\cat{T} = \add(T)$ of a module category $\cat{C}$, so one may 
choose $\cat{X} = \cat{T}$ in \cref{theoremA} (see \cref{example:d-CT}). 
Then, for a short exact sequence
$
\delta\colon\;
\begin{tikzcd}[column sep=0.5cm, cramped]
A_{2}
    \arrow[tail]{r}{\partial^{A}_{2}}
&A_{1}
    \arrow[two heads]{r}{\partial^{A}_{1}}
&A_{0}
\end{tikzcd}
$
in $\cat{C}$, 
our term $\theta_{\cat{T}}([\Cok\left(\restr{\cat{C}(-,\partial^{A}_{1})}{\cat{T}}\right)])$
specialises to 
the term $\kappa^{-1}([\delta^{*}(T)]_{\Lambda})$ of \cite[Thm.~C]{Reid-Modules-determined-by-their-composition-factors-in-higher-homological-algebra}, using that 
our index $[-]_{\cat{T}}$ corresponds to Reid's $\Ind{\cat{T}}(-)$ (which is just $\indxx{\cat{T}}(-)$ in our notation above)  
via the isomorphism given in \cref{thm_dct_abelian}.

A natural choice of $(\cat{C},\adj{E})$ as in \cref{theoremA} is a $d$-abelian category. However, there are many examples of such $(\cat{C},\adj{E})$ that are not $d$-abelian. For instance, any $d$-torsion class of a Krull-Schmidt $d$-abelian category that embeds in a finite length abelian category can serve as $(\cat{C},\adj{E})$ (see \cref{example:d-torsion-free-have-d-kernels}). 
We study $d$-exact categories that have $d$-kernels in more detail in Appendix~\ref{appendix}, where we also give several examples.

\section{Background and notation}\label{background}

In this section, we present a summary on the module category of an additive category and recall some results we will apply later. 
Let $\cat{C}$ denote an additive category. 
We say that $\cat{C}$ \textit{has weak kernels} if, for every morphism $b\colon  B\rightarrow C$ in $\cat{C}$, there is a  morphism $a\colon  A\rightarrow B$ in $\cat{C}$ inducing an exact sequence as follows. 
\[
\begin{tikzcd}[column sep=1.5cm]
\cat{C}(-,A)
    \arrow{r}{\cat{C}(-,a)}
& \cat{C}(-,B)
    \arrow{r}{\cat{C}(-,b)}
& \cat{C}(-,C)
\end{tikzcd}
\]

We give a brief overview of the category of $\cat{C}$-modules and its subcategory of finitely presented $\cat{C}$-modules.

\begin{defn}\label{defn_modC}
    Suppose $\cat{C}$ is a skeletally small, idempotent complete, additive category that has weak kernels. 
    Let $\Ab$ denote the category of all abelian groups.
    \begin{enumerate}
    
    \item\label{ModC} The \textit{category of $\cat{C}$-modules}, denoted by $\rMod{\cat{C}}$, is the abelian category of all (covariant) additive functors $\fun{M}\colon \cat{C}^{\op}\rightarrow \Ab$.

    \item\label{modC} A $\cat{C}$-module $\fun{M}\colon \cat{C}^{\op}\rightarrow \Ab$ is \textit{finitely presented} if there is an exact sequence 
    \[
    \begin{tikzcd}
        \cat{C}(-,A)
            \arrow{r}
        &\cat{C}(-,B)
            \arrow{r}
        &\fun{M}
            \arrow{r}
        &0
    \end{tikzcd}
    \]
    in $\rMod{\cat{C}}$ 
    for some objects $A,B\in\cat{C}$ (see \cite[p.~155]{Beligiannis-on-the-freyd-cats-of-additive-cats}). We denote  by $\rmod{\cat{C}}$ the full subcategory of $\rMod{\cat{C}}$ consisting of the finitely presented $\cat{C}$-modules.
   Under the assumptions on $\cat{C}$, we have that $\rmod{\cat{C}}$ is abelian and the inclusion functor $\rmod{\cat{C}} \to \rMod{\cat{C}}$ is exact (see e.g.\  \cite[Sec.~III.2]{Auslander-representation-dimension-of-artin-algebras-QMC}). 
   
   \item\label{Yoneda} Under the assumptions on $\cat{C}$, the \textit{Yoneda embedding} $\yoneda \colon  \cat{C}\rightarrow \rmod{\cat{C}}$ given by 
   $A\mapsto \yoneda A\deff\cat{C}(-,A)$ 
   and 
   $(A\overset{a}{\to}B)\mapsto \yoneda a \deff \cat{C}(-,a)$ 
   is fully faithful, additive, and its values are, up to isomorphism, all the projective objects in $\rmod{\cat{C}}$ (and in $\rMod{\cat{C}}$). 
    See \cite[Prop.~2.2]{Auslander-Rep-theory-of-Artin-algebras-I}.
    \end{enumerate}
\end{defn}

We will further assume that $\cat{C}$ has a nice enough subcategory. We recall the definitions of the needed properties.

\begin{defn}\label{def:additive-approxs}
    By an \textit{additive subcategory} $\cat{X}$ of an additive category $\cat{C}$ we mean a full subcategory that is closed under isomorphisms, finite direct sums and direct summands.

    Let $\cat{X}\subseteq \cat{C}$ be a full subcategory and $C$ an object in $\cat{C}$. A morphism $x\colon  X\rightarrow C$ with $X\in\cat{X}$ is called a \textit{right $\cat{X}$-approximation of $C$} if, for each $X'\in\cat{X}$, the induced morphism $\cat{C}(X',x)\colon \cat{C}(X',X)\onto \cat{C}(X',C)$ is surjective. If each $C\in\cat{C}$ has a right $\cat{X}$-approximation, then $\cat{X}$ is said to be \textit{contravariantly finite} (see \cite[pp.~113--114]{auslander-reiten-contravariantly}).
\end{defn}

For the rest of this section, we work in the following setup.

\begin{setup}\label{setup:section2}
    Let $\cat{C}$ be a skeletally small, idempotent complete, additive category that has weak kernels. In addition, suppose  $\cat{X}\subseteq \cat{C}$ is a contravariantly finite, additive subcategory.
\end{setup}

We recall the connection between the categories of $\cat{X}$-modules and of $\cat{C}$-modules.

\begin{rem}
     First note that the assumptions on $\cat{X}$ ensure it is itself a skeletally small, idempotent complete, additive category with weak kernels. Hence, \cref{defn_modC} holds replacing $\cat{C}$ by $\cat{X}$.
     Moreover, there is a triplet of adjoint functors
    \begin{equation}\label{eqn_adjoints}
    \begin{tikzcd}[column sep=3cm]
        \rMod{\cat{C}} 
            \arrow{r}{\restr{(-)}{\cat{X}}}
        & \rMod{\cat{X}},
            \arrow[bend right]{l}[swap]{\adj{L}}
            \arrow[bend left]{l}{\adj{R}}
    \end{tikzcd}
    \end{equation}
    where the left adjoint $\adj{L}$ and the right adjoint $\adj{R}$ of the restriction functor $\restr{(-)}{\cat{X}}$ are fully faithful functors. Furthermore, $\adj{L}$ is right exact and $\adj{R}$ is left exact, while $\restr{(-)}{\cat{X}}$ is exact. See \cite[Props.~3.1 and 3.4]{Auslander-Rep-theory-of-Artin-algebras-I}. 
\end{rem}

We conclude this section by recalling the following result that we will use in the proof of the main result in \cref{sec:defining-thetaX}. In the following, let $\iota$ denote the inclusion of $\Ker \restr{(-)}{\cat{X}}$ into $\rmod{\cat{C}}$. Then the induced homomorphism on the Grothendieck groups is
\[
    \begin{tikzcd}[column sep=1.8cm]
        K_0(\Ker \restr{(-)}{\cat{X}})
            \arrow{r}{K_{0}(\iota)}
        &K_0(\rmod{\cat{C}}),
    \end{tikzcd}
    \]
    where $K_{0}(\iota)([\fun{M}]) = [\fun{M}]$ for $\fun{M}\in \Ker \restr{(-)}{\cat{X}}$.

\begin{lem}[{\cite[Lem.~2.3 and Prop.~2.6]{fedele2024index}}]\label{lemmas_FJS}
    The exact functor $\restr{(-)}{\cat{X}}\colon  \rMod{\cat{C}}\rightarrow \rMod{\cat{X}}$ restricts to an exact functor on finitely presented modules $\restr{(-)}{\cat{X}}\colon \rmod{\cat{C}}\rightarrow \rmod{\cat{X}}$.
    Moreover, the latter induces an exact sequence of Grothendieck groups
    \[
    \begin{tikzcd}[column sep=1.8cm]
        K_0(\Ker \restr{(-)}{\cat{X}})
            \arrow{r}{K_{0}(\iota)}
        &K_0(\rmod{\cat{C}})
            \arrow{r}{K_0(\restr{(-)}{\cat{X}})} 
        &K_0(\rmod{\cat{X}})
            \arrow{r}
        &0.
    \end{tikzcd}
    \]
    Consequently, the subgroup $\Ker K_0(\restr{(-)}{\cat{X}})$ of $K_0(\rmod{\cat{C}})$ is generated by 
    \begin{align*}
        \Set{ [\fun{M}] | \fun{M}\in\rmod{\cat{C}} \text{ and } \restr{\fun{M}}{\cat{X}}=0 }.
    \end{align*}
\end{lem}

\section{\texorpdfstring{$d$}{d}-exact categories}
\label{sec:d-exact-categories}

Let $d\geq 1$ be an integer. Abelian and exact categories are central in homological algebra. Their $d$-dimensional analogues have been introduced in the development of higher homological algebra, and they are known as $d$-abelian and $d$-exact categories \cite{Jasso-n-abelian-and-n-exact-categories}. Just like in the classical theory, each $d$-abelian category has a canonical $d$-exact structure \cite[Thm.~4.4]{Jasso-n-abelian-and-n-exact-categories} and we focus on $d$-exact categories in this article. 
Throughout this section, let $\cat{C}$ denote an additive category.

\begin{defn}\label{def:d-kernels}
\cite[Def.~2.2]{Jasso-n-abelian-and-n-exact-categories} 
Suppose $\partial^{B}_{1}\colon B_{1} \to B_{0}$ is a morphism in $\cat{C}$. A \emph{$d$-kernel} in $\cat{C}$ of $\partial^{B}_{1}$ is a sequence
\[
(\partial^{B}_{d+1},\ldots,\partial^{B}_{2}) \colon\;\;
\begin{tikzcd}
B_{d+1}
    \arrow{r}{\partial^{B}_{d+1}}
&B_{d}
    \arrow{r}{\partial^{B}_{d}}
&\cdots
    \arrow{r}{\partial^{B}_{3}}
&B_{2}
    \arrow{r}{\partial^{B}_{2}}
&B_{1}
\end{tikzcd}
\]
of $d$ composable morphisms in $\cat{C}$, such that the induced sequence 
\[
\begin{tikzcd}[column sep=1.3cm]
0
    \arrow{r}
&\yoneda B_{d+1}
    \arrow{r}{\yoneda \partial^{B}_{d+1}}
&\yoneda B_{d}
    \arrow{r}{\yoneda \partial^{B}_{d}}
&\cdots
    \arrow{r}{\yoneda \partial^{B}_{2}}
&\yoneda B_{1}
    \arrow{r}{\yoneda \partial^{B}_{1}}
&\yoneda B_{0}
\end{tikzcd}
\]
in $\rMod{\cat{C}}$ is exact. 
In this case, we say that 
\begin{equation}\label{eqn:left-d-exact}
\begin{tikzcd}
B_{d+1}
    \arrow[tail]{r}{\partial^{B}_{d+1}}
&B_{d}
    \arrow{r}{\partial^{B}_{d}}
&\cdots
    \arrow{r}{\partial^{B}_{3}}
&B_{2}
    \arrow{r}{\partial^{B}_{2}}
&B_{1}
    \arrow{r}{\partial^{B}_{1}}
&B_{0}
\end{tikzcd}
\end{equation}
is a \emph{left $d$-exact sequence}.

One defines a \emph{$d$-cokernel} and a \emph{right $d$-exact sequence} dually.
\end{defn}

For a left $d$-exact sequence \eqref{eqn:left-d-exact}, it follows from the definition that $\partial^{B}_{d+1}$ is a monomorphism in $\cat{C}$. We note that left $d$-exact sequences were called `$d$-kernel sequences' in \cite[Def.~4.9]{HerschendLiuNakaoka-n-exangulated-categories-I-definitions-and-fundamental-properties}.

\begin{defn}\label{def:has-d-kernels}
We say that $\cat{C}$ \emph{has $d$-kernels} if, for each morphism $\partial^{B}_{1}\colon B_{1}\to B_{0}$ in $\cat{C}$, there is a left $d$-exact sequence of the form \eqref{eqn:left-d-exact}.
\end{defn}

\begin{defn}\label{def:d-exact-sequence}
\cite[Def.~2.4]{Jasso-n-abelian-and-n-exact-categories} 
A sequence
\begin{equation}\label{eqn:d-exact-sequence}
\begin{tikzcd}
B_{d+1}
    \arrow[tail]{r}{\partial^{B}_{d+1}}
&B_{d}
    \arrow{r}{\partial^{B}_{d}}
&\cdots
    \arrow{r}{\partial^{B}_{3}}
&B_{2}
    \arrow{r}{\partial^{B}_{2}}
&B_{1}
    \arrow[two heads]{r}{\partial^{B}_{1}}
&B_{0}
\end{tikzcd}
\end{equation}
of morphisms in $\cat{C}$ is called \emph{$d$-exact} if it is both left $d$-exact and right $d$-exact, i.e.\ 
$(\partial^{B}_{d+1},\ldots,\partial^{B}_{2})$ is a $d$-kernel of $\partial^{B}_{1}$ and $(\partial^{B}_{d},\ldots,\partial^{B}_{1})$ is a $d$-cokernel of $\partial^{B}_{d+1}$. 
We denote the complex \eqref{eqn:d-exact-sequence} by $B_{\bullet}$.
\end{defn}

\begin{rem}
We emphasise here that the definition of a left $d$-exact sequence is purely about intrinsic properties of $\cat{C}$; indeed, for the left $d$-exact sequence \eqref{eqn:left-d-exact} to be left exact means precisely that the first $d$ morphisms $(\partial^{B}_{d+1},\ldots,\partial^{B}_{2}) $ constitute a $d$-kernel of the rightmost morphism $\partial^{B}_{1}$, and nothing more. Similarly for a (right) $d$-exact sequence. 
In particular, there is no reference to a $d$-exact structure on $\cat{C}$ (cf.\ \cref{def:d-exact-category} below).
\end{rem}

We now recall the definition of a $d$-exact category. However, we omit the details we will not use in the sequel. For complete definitions, see \cite[Sec.~4]{Jasso-n-abelian-and-n-exact-categories} or \cite[Def.~4.19]{HerschendLiuNakaoka-n-exangulated-categories-I-definitions-and-fundamental-properties}.

\begin{defn}\label{def:d-exact-category}
\cite[Def.~4.2]{Jasso-n-abelian-and-n-exact-categories}
Suppose $\adj{E}$ is a class of $d$-exact sequences in the additive category $\cat{C}$. If $B_{\bullet}\in\adj{E}$, then $B_{\bullet}$ is called \emph{$\adj{E}$-admissible}, the morphism $\partial^{B}_{d+1}$ is called an \emph{$\adj{E}$-admissible inflation} and $\partial^{B}_{1}$ an \emph{$\adj{E}$-admissible deflation}. The pair $(\cat{C},\adj{E})$ is a \emph{$d$-exact} category if the following axioms are satisfied.
\begin{enumerate}[label=\textup{(E\arabic*)}]
\setcounter{enumi}{-1}
    \item[(EC)]  The class $\adj{E}$ is closed under weak isomorphisms (see \cite[Def.~4.1]{Jasso-n-abelian-and-n-exact-categories}).
    \item The zero complex $0\to \cdots \to 0$ belongs to $\adj{E}$.
    \item The class of $\adj{E}$-admissible inflations is closed under composition.
    \item[(E$1^{\op}$)] Dually, the class of $\adj{E}$-admissible deflations is closed under composition. 
    \item Dual of (E$2^{\op}$) below.
    \item[(E$2^{\op}$)] For each $B_{\bullet}\in\adj{E}$ and each morphism $g\colon C_{0} \to B_{0}$, there is a $d$-pullback diagram (see \cite[dual of Def.~2.11]{Jasso-n-abelian-and-n-exact-categories})
    \begin{equation}\label{eqn:d-pullback-diagram}
    \begin{tikzcd}
    (C_{d+1}
        \arrow[tail]{r}{\partial^{C}_{d+1}}
        \arrow[equals]{d}{}
    &)\;C_{d}
        \arrow{r}{\partial^{C}_{d}}
        \arrow{d}{f_{d}}
    &\cdots
        \arrow{r}
    &C_{1}
        \arrow[two heads]{r}{\partial^{C}_{1}}
        \arrow{d}{f_{1}}
    &C_{0}
        \arrow{d}{g}
    \\
    (B_{d+1}
        \arrow[tail]{r}{\partial^{B}_{d+1}}
    &)\;B_{d}
        \arrow{r}{\partial^{B}_{d}}
    &\cdots
        \arrow{r}
    &B_{1}
        \arrow[two heads]{r}{\partial^{B}_{1}}
    &B_{0},
    \end{tikzcd}
    \end{equation}
    such that $C_{\bullet}\in\adj{E}$.
\end{enumerate}
\end{defn}

\begin{rem}
The axiom (E$2^{\op}$) (and hence also (E$2$)) we give above differ by a small subtlety compared to the original in \cite[Def.~4.2]{Jasso-n-abelian-and-n-exact-categories}. However, nothing is lost due to \cite[Prop.~4.8]{Jasso-n-abelian-and-n-exact-categories}. We use the version stated above for simplicity in our exposition below.
\end{rem}

\subsection{Relative theory via \texorpdfstring{$d$}{d}-exangulated categories}

Herschend--Liu--Nakaoka introduced $d$-exangulated categories in \cite{HerschendLiuNakaoka-n-exangulated-categories-I-definitions-and-fundamental-properties}, giving a simultaneous generalisation of $d$-exact and $(d+2)$-angulated categories. Since we do not use the specific mechanics of $d$-exangulated categories here, we omit the details. 
However, the relative theory of \cite[Sec.~3.2]{HerschendLiuNakaoka-n-exangulated-categories-I-definitions-and-fundamental-properties} allows us to produce a $d$-exact substructure of a skeletally small $d$-exact category as follows.  
When $d=1$, this is done in \cite[Sec.~1.2]{DraxlerReitenSmaloSolberg-exact-categories-and-vector-space-categories}.

\begin{prop}\label{prop:relative-d-exact-subcategory}
Suppose $(\cat{C},\adj{E})$ is a skeletally small $d$-exact category and that $\cat{X}\sse\cat{C}$ is a full subcategory. 
Define $\adj{E}_{\cat{X}} \sse \adj{E}$ to be all the $\adj{E}$-admissible $d$-exact sequences $B_{\bullet}$ 
such that 
$\restr{(\yoneda\partial^{B}_{1})}{\cat{X}}
    \colon 
        \restr{(\yoneda B_{1})}{\cat{X}} 
            \onto 
        \restr{(\yoneda B_{0})}{\cat{X}}$ 
is an epimorphism. 
Then $(\cat{C},\adj{E}_{\cat{X}})$ is a skeletally small $d$-exact category. 
\end{prop}

\begin{proof}
By \cite[Prop.~4.34]{HerschendLiuNakaoka-n-exangulated-categories-I-definitions-and-fundamental-properties}, we can view $(\cat{C},\adj{E})$ as a $d$-exangulated category $(\cat{C},\BE,\fs)$. 
The biadditive functor 
$\BE\colon \cat{C}^{\op}\times \cat{C}\to \Ab$ is given as follows.
\begin{enumerate}
    \item For $A,C\in\cat{C}$, define
$\BE(C,A) 
    \deff \Set{ [B_{\bullet}] | B_{\bullet}\in\adj{E},\, B_{d+1} = A \text{ and } B_{0} = C }, 
$ where $[B_{\bullet}]$ denotes the (Yoneda) equivalence class of the $d$-exact sequence $B_{\bullet}$; see the homotopy equivalence relation defined for complexes with fixed end terms in \cite[p.~540]{HerschendLiuNakaoka-n-exangulated-categories-I-definitions-and-fundamental-properties}. 
Equipped with the Baer sum (see \cite[Def.~4.28]{HerschendLiuNakaoka-n-exangulated-categories-I-definitions-and-fundamental-properties}), 
the set $\BE(C,A)$ is an abelian group with identity element being the class
\[
0 = 
[
\begin{tikzcd}
A 
    \arrow{r}{\id{A}}
&A 
    \arrow{r}
&0 
    \arrow{r}
&\cdots
    \arrow{r}
&0
    \arrow{r}
&C 
    \arrow{r}{\id{C}}
&C
\end{tikzcd}
].
\]

    \item For a morphism $g\colon C_{0}\to B_{0}$ in $\cat{C}$, define 
$\BE(g,B_{d+1}) \colon \BE(B_{0},B_{d+1}) \to \BE(C_{0},B_{d+1})$
by 
$
\BE(g,B_{d+1})([B_{\bullet}]) 
\deff [C_{\bullet}]
$
whenever 
$[B_{\bullet}]\in \BE(B_{0},B_{d+1})$ 
and there is a $d$-pullback diagram \eqref{eqn:d-pullback-diagram}. The mapping 
$\BE(B_{0},h)\colon \BE(B_{0},B_{d+1}) \to \BE(B_{0},D_{d+1})$
is defined dually for a morphism $h\colon B_{d+1}\to D_{d+1}$.
\end{enumerate}
The exact realisation $\fs$ is simply given by $\fs([B_{\bullet}]) \deff [B_{\bullet}]$; see \cite[Def.~2.22]{HerschendLiuNakaoka-n-exangulated-categories-I-definitions-and-fundamental-properties}.
(For more details, see \cite[Sec.~4.3]{HerschendLiuNakaoka-n-exangulated-categories-I-definitions-and-fundamental-properties}.)
Importantly, for a complex $B_{\bullet}$, we have
$[B_{\bullet}]\in\BE(B_{0},B_{d+1})$,  
if and only if 
$B_{\bullet}$ is part of a distinguished $d$-exangle in $(\cat{C},\BE,\fs)$, 
if and only if 
$B_{\bullet}$ is an $\adj{E}$-admissible $d$-exact sequence.

Now we use the relative theory for $d$-exangulated categories. Define a subfunctor $\BE_{\cat{X}}$ of $\BE$ by 
\[
\BE_{\cat{X}}(C,A) 
    \deff \Set{[B_{\bullet}]\in\BE(C,A) | \forall g\colon X\to C\text{ with }X\in\cat{X}\text{, we have }\BE(g,A)([B_{\bullet}]) = 0},
\]
and $\BE_{\cat{X}}$ is just the restriction of $\BE$ on morphisms.
Note that $\BE_{\cat{X}}$ is the first bifunctor defined in \cite[Def.~3.18]{HerschendLiuNakaoka-n-exangulated-categories-I-definitions-and-fundamental-properties} with $\cat{I} = \cat{X}$, and hence $(\cat{C},\BE_{\cat{X}},\fs_{\cat{X}})$ is a $d$-exangulated category by \cite[Props.~3.16 and 3.19]{HerschendLiuNakaoka-n-exangulated-categories-I-definitions-and-fundamental-properties}, where $\fs_{\cat{X}}\deff \restr{\fs}{\BE_{\cat{X}}}$.
It follows from \cite[Rem.~4.38]{HerschendLiuNakaoka-n-exangulated-categories-I-definitions-and-fundamental-properties} that $(\cat{C},\BE_{\cat{X}},\fs_{\cat{X}})$ is also a $d$-exact category (see also \cite[Cor.~4.12]{Klapproth-n-extension-closed}). 
That is, there is a (skeletally small) $d$-exact category $(\cat{C},\adj{E}')$ such that 
$[B_{\bullet}]\in\BE_{\cat{X}}(B_{0},B_{d+1})$ 
if and only if 
$B_{\bullet}$ lies in $\adj{E}'$.

Therefore, it suffices to show that 
$\adj{E}' = \adj{E}_{\cat{X}}$. 
This is a consequence of the following identities, where the first equality is by 
\cite[Claim 2.15]{HerschendLiuNakaoka-n-exangulated-categories-I-definitions-and-fundamental-properties}: 
\begin{align*}
\BE_{\cat{X}}(B_{0},B_{d+1}) 
    &= \Set{[B_{\bullet}]\in\BE(B_{0},B_{d+1}) | 
    \begin{array}{l}
        \forall g\colon X\to B_{0}\text{ with }X\in\cat{X}\text{, there exists}\\
        h\colon X\to B_{1}\text{ with }g = \partial^{B}_{1}h
    \end{array}}\\
    &= \Set{[B_{\bullet}]\in\BE(B_{0},B_{d+1}) | 
    \begin{array}{l}
    \cat{C}(X,\partial^{B}_{1})
    \colon \cat{C}(X,B_{1}) \onto \cat{C}(X,B_{0})\\
    \text{is surjective for each }
    X\in\cat{X}
    \end{array}}\\
    &= \Set{[B_{\bullet}]\in\BE(B_{0},B_{d+1}) | 
    \begin{array}{l}
    \restr{(\yoneda\partial^{B}_{1})}{\cat{X}}
    \colon  \restr{(\yoneda {B}_{1})}{\cat{X}} \onto  \restr{(\yoneda {B}_{0})}{\cat{X}}\\
    \text{is an epimorphism in }\rmod{\cat{X}}
    \end{array}}.
\end{align*}
\end{proof}

\subsection{The Grothendieck group of a \texorpdfstring{$d$}{d}-exact category}

Suppose $\cat{C}$ is now also skeletally small. Denote by $\Iso{(\cat{C})}$ the set of isomorphism classes of objects in $\cat{C}$. 
For an object $A\in\cat{C}$, we denote the isomorphism class of $A$ by $[A]\in\Iso{(\cat{C})}$. 
The \emph{split Grothendieck group} $K_{0}^{\sp}(\cat{C})$ of $\cat{C}$ is the free abelian group generated by $\Iso{(\cat{C})}$ modulo the subgroup 
$\Braket{[A] - [B] + [C] | A\to B \to C \text{ is a split short exact sequence in }\cat{C}}$. 
We denote the class in $K_{0}^{\sp}(\cat{C})$ of an object $A\in\cat{C}$ by $[A]^{\sp}$.

\begin{defn}\label{def:grothendieck-group}
The \emph{Grothendieck group} of a skeletally small $d$-exact category $(\cat{C},\adj{E})$ is the abelian group 
\[
\left.K_{0}(\cat{C},\adj{E})
    \deff
K_{0}^{\sp}(\cat{C}) \middle/ 
    \Braket{ \sum_{i=0}^{d+1}(-1)^{i}[B_{i}]^{\sp} | B_{\bullet}\in\adj{E} }.\right.
\]
The class of $A\in\cat{C}$ in $K_{0}(\cat{C},\adj{E})$ is denoted $[A]$.
\end{defn}

By \cref{prop:relative-d-exact-subcategory}, given a full subcategory $\cat{X}$ of a skeletally small $d$-exact category $(\cat{C},\adj{E})$, there is an induced $d$-exact subcategory $(\cat{C},\adj{E}_{\cat{X}})$ of $(\cat{C},\adj{E})$. 
Hence, we may consider the Grothendieck group 
$K_{0}(\cat{C},\adj{E}_{\cat{X}})$ and the canonical surjection
\begin{equation}\label{eqn:pi_X}
\begin{tikzcd}
\pi_{\cat{X}}\colon K_{0}^{\sp}(\cat{C})
    \arrow[two heads]{r}
& K_{0}(\cat{C},\adj{E}_{\cat{X}}).
\end{tikzcd}
\end{equation}
For an object $C\in\cat{C}$, we denote its class in $K_{0}(\cat{C},\adj{E}_{\cat{X}})$ by $[C]_{\cat{X}}$, so that the homomorphism $\pi_{\cat{X}}$ of abelian groups is given by 
$\pi_{\cat{X}}([C]^{\sp}) = [C]_{\cat{X}}$ on generators.

\section{Main results}
\label{sec:main-results}

For the main results in this section, we assume \cref{setup:section4} (see below), which includes that $(\cat{C},\adj{E})$ is an idempotent complete $d$-exact category and $\cat{X}$ is a contravariantly finite (see \cref{def:additive-approxs}), \emph{generating} subcategory of $\cat{C}$. When $d=1$, \cref{def:generating} below recovers that of a generating subcategory of an exact category as defined in \cite[Def.~5.1]{HenrardKvammevanRoosmalen2022}.

\begin{defn}\label{def:generating}
Let $(\cat{C},\adj{E})$ be a $d$-exact category. We call a subcategory $\cat{X}$ of $\cat{C}$ \emph{generating} if, for each $C\in\cat{C}$, there is an $\adj{E}$-admissible deflation $X \onto C$ for some $X\in\cat{X}$. 
\end{defn}

In this section, we work under the following setup.

\begin{setup}\label{setup:section4}
Let $(\cat{C},\adj{E})$ be a skeletally small, idempotent complete, $d$-exact category that has $d$-kernels. In addition, let $\cat{X}$ be a contravariantly finite, generating, additive subcategory of $\cat{C}$.
\end{setup}

By \cref{defn_modC}\ref{modC}, we know that $\rmod{\cat{C}}$ is a skeletally small abelian category, and hence a skeletally small $1$-exact category. The Grothendieck group $K_{0}(\rmod{\cat{C}})$ is thus defined as in \cref{def:grothendieck-group}; it is the split Grothendieck group of $\rmod{\cat{C}}$ modulo relations arising from short exact sequences of functors in $\rmod{\cat{C}}$. Similarly for $K_{0}(\rmod{\cat{X}})$.

We prove \cref{theoremA} in \cref{sec:defining-thetaX}. 
Our strategy is to first produce a homomorphism 
$\theta_{\cat{C}}\colon K_{0}(\rmod{\cat{C}})\to K_{0}^{\sp}(\cat{C})$ (see \cref{thm:thetaC} in \cref{sec:defining-thetaC}), and then use that $K_{0}(\rmod{\cat{X}})$ is the cokernel of $K_{0}(\iota)$ as in \cref{lemmas_FJS} to obtain $\theta_{\cat{X}}\colon K_{0}(\rmod{\cat{X}})\to K_{0}(\cat{C},\adj{E}_{\cat{X}})$. 

Before we give some examples to see that \cref{setup:section4} is reasonable, we observe the following, where an additive category $\cat{C}$ is \textit{weakly idempotent complete} if every split epimorphism has a kernel (or equivalently, every split monomorphism has a cokernel).

\begin{lem}\label{lem:precovering-and-generating-iff-Gen}
Suppose $(\cat{C},\adj{E})$ is a $d$-exact category and $\cat{X}\sse\cat{C}$ is a subcategory. If $d=1$, assume further that $\cat{C}$ is weakly idempotent complete. 
Then $\cat{X}$ is contravariantly finite and generating, if and only if $\cat{X}$ satisfies: 
\begin{enumerate}[label=\textup{(\textbf{Gen})}]
    \item\label{gen} For each $C\in\cat{C}$, there exists a right $\cat{X}$-approximation $X\onto C$ of $C$ that is also an $\adj{E}$-admissible deflation.
\end{enumerate}
\end{lem}

\begin{proof}
Suppose $\cat{X}$ is contravariantly finite and generating in $\cat{C}$, so that for $C\in\cat{C}$ there is a right $\cat{X}$-approximation $x\colon X\to C$ and an $\adj{E}$-admissible deflation $x'\colon X'\onto C$ with $X'\in\cat{X}$. This implies that there exists a morphism $y\colon X'\to X$ with $xy = x'$.  Since $\cat{C}$ is weakly idempotent complete, it follows from \cite[dual of Cor.~2.5]{Klapproth-n-extension-closed} that $x\colon X \onto C$ is also an $\adj{E}$-admissible deflation. 
Hence, $\cat{X}$ satisfies \ref{gen}. 

For the converse there is nothing to show. 
\end{proof}

Note that in \cref{lem:precovering-and-generating-iff-Gen}, the implication 
\begin{align*}
\cat{X}\text{ satisfies \ref{gen} }\Longrightarrow\cat{X}\text{ is contravariantly finite and generating}
\end{align*}
always holds 
(even when $d=1$ we need not assume $\cat{C}$ is weakly idempotent complete).

\begin{rem}
When $d=1$, the condition \ref{gen} was called `strongly contravariantly finite' in e.g.\ \cite[Def.~3.19]{ZhouZhu-triangulated-quotient-categories-revisited}, \cite[p.~6]{FaberMarshPressland-reduction-of-frobenius-extriangulated-categories}. 
However, `strongly' has also been used previously (e.g.\ in \cite[Def.~4.3]{HerschendJorgensenVaso-wide-subcategories-of-d-CT-subcategories}) to indicate a uniqueness in factorisations for approximations. Therefore, we avoid this terminology here. 
\end{rem}

Now we recall some contexts in which generating subcategories appear.

\begin{example}\label{example:generating}
Consider an exact category $(\cat{C},\adj{E})$ where $\cat{C}$ is abelian and $\cat{E}$ is the exact structure consisting of all the short exact sequences in $\cat{C}$. 
In this case, since the $\adj{E}$-admissible deflations are precisely the epimorphisms in $\cat{C}$, a subcategory $\cat{X}\sse\cat{C}$ is generating in the sense of Definition~\ref{def:generating} if and only if, for each $C\in\cat{C}$, there is an epimorphism $X\onto C$ in $\cat{C}$ for some $X\in\cat{X}$ (see also \cite[p.~724]{Jasso-n-abelian-and-n-exact-categories}). 

As an explicit example, let $R$ be a ring and consider the abelian category $\rMod{R}$ of right $R$-modules. The subcategory $\rProj{R}$ of projective right $R$-modules is generating in $\rMod{R}$. 
\end{example}

\begin{example}\label{example:d-CT}
By definition, any $d$-cluster tilting subcategory of an abelian category is contravariantly finite and generating (see \cite[Def.~3.14]{Jasso-n-abelian-and-n-exact-categories}). 
More generally, any $d$-cluster tilting subcategory of an exact category satisfies \ref{gen} by definition (see \cite[Def.~4.13(ii)]{Jasso-n-abelian-and-n-exact-categories}), and hence is contravariantly finite and generating by \cref{lem:precovering-and-generating-iff-Gen}.
\end{example}

\subsection{Defining \texorpdfstring{$\theta_{\cat{C}}$}{thetaC}}\label{sec:defining-thetaC}

We define $\theta_{\cat{C}}\colon K_{0}(\rmod{\cat{C}})\to K_{0}^{\sp}(\cat{C})$ via three lemmas, which are analogues of the results in \cite[Sec.~4.1]{fedele2024index}. 
We start by showing that any finitely presented $\cat{C}$-module $\fun{M}$ has projective dimension at most $d+1$ in $\rmod{\cat{C}}$.

\begin{lem}\label{lem:projective-resolution}
For each $\fun{M}\in\rmod{\cat{C}}$, there is a left $d$-exact sequence 
\begin{equation}\label{eqn:left-d-exact-A}
\begin{tikzcd}
A_{d+1}
    \arrow[tail]{r}{\partial^{A}_{d+1}}
&A_{d}
    \arrow{r}{\partial^{A}_{d}}
&\cdots
    \arrow{r}{\partial^{A}_{2}}
&A_{1}
    \arrow{r}{\partial^{A}_{1}}
&A_{0}
\end{tikzcd}
\end{equation}
in $\cat{C}$, such that the following induced sequence is exact in $\rmod{\cat{C}}$. 
\begin{equation}\label{eqn:proj-res-for-M}
\begin{tikzcd}[column sep=0.9cm]
0
    \arrow{r}
&\yoneda A_{d+1}
    \arrow{r}{\yoneda \partial^{A}_{d+1}}
&\yoneda A_{d}
    \arrow{r}{\yoneda \partial^{A}_{d}}
&\cdots
    \arrow{r}{\yoneda \partial^{A}_{2}}
&\yoneda A_{1}
    \arrow{r}{\yoneda \partial^{A}_{1}}
&\yoneda A_{0}
    \arrow{r}
& \fun{M}
    \arrow{r}
&0
\end{tikzcd}
\end{equation}
\end{lem}

\begin{proof}
If $\fun{M}\in\rmod{\cat{C}}$, then there is an exact sequence 
$
\begin{tikzcd}
\yoneda A_{1}
    \arrow{r}{\yoneda \partial^{A}_{1}}
&\yoneda A_{0}
    \arrow{r}
& \fun{M} 
    \arrow{r}
&0
\end{tikzcd}
$
for some morphism $\partial^{A}_{1} \colon A_{1} \to A_{0}$ in $\cat{C}$ (see \cref{defn_modC}). 
Since $\cat{C}$ has $d$-kernels, we obtain a left $d$-exact sequence of the form \eqref{eqn:left-d-exact-A}, which induces the exact sequence \eqref{eqn:proj-res-for-M}.
\end{proof}

Note that \eqref{eqn:left-d-exact-A} is not necessarily $\cat{E}$-admissible. Moreover, it follows from \cref{lem:projective-resolution} that each $\fun{M}\in\rmod{\cat{C}}$ fits in an exact sequence of the form \eqref{eqn:proj-res-for-M} and in this case $M$ is isomorphic to $\Cok(\yoneda \partial^{A}_{1})$.
Over the next two lemmas, we will show that the assignment 
$
\Cok(\yoneda \partial^{A}_{1}) 
    \overset{\theta_{\cat{C}}}{\longmapsto}
    \sum_{i=0}^{d+1}(-1)^{i}[A_{i}]^{\sp}
$
is well-defined on $K_{0}(\rmod{\cat{C}})$.

\begin{lem}\label{lem:isoclasses}
Assume that \eqref{eqn:left-d-exact-A} and 
\begin{equation}\label{eqn:left-d-exact-B}
\begin{tikzcd}
B_{d+1}
    \arrow[tail]{r}{\partial^{B}_{d+1}}
&B_{d}
    \arrow{r}{\partial^{B}_{d}}
&\cdots
    \arrow{r}{\partial^{B}_{2}}
&B_{1}
    \arrow{r}{\partial^{B}_{1}}
&B_{0}
\end{tikzcd}
\end{equation}
are left $d$-exact sequences in $\cat{C}$ with 
$\Cok(\yoneda \partial^{A}_{1}) \cong \Cok(\yoneda \partial^{B}_{1})$. 
Then in $K_{0}^{sp}(\cat{C})$ we have 
\[
\sum_{i=0}^{d+1}(-1)^{i}[A_{i}]^{\sp}
    = \sum_{i=0}^{d+1}(-1)^{i}[B_{i}]^{\sp}.
\]
\end{lem}

\begin{proof}
Since $\yoneda A_{i}$ and $\yoneda B_{i}$ are projective objects in $\rmod{\cat{C}}$, by \cite[Lem.~IX.6.3]{Aluffi-Chapter0} we obtain the following commutative diagram with exact rows in the abelian category $\rmod{\cat{C}}$.
\[
\begin{tikzcd}[column sep=1cm]
0
    \arrow{r}
&[-0.4cm]\yoneda A_{d+1}
    \arrow{r}{\yoneda \partial^{A}_{d+1}}
    \arrow[dotted]{d}
&[0.2cm]\yoneda A_{d}
    \arrow{r}{\yoneda \partial^{A}_{d}}
    \arrow[dotted]{d}
&\cdots
    \arrow{r}{\yoneda \partial^{A}_{2}}
&\yoneda A_{1}
    \arrow{r}{\yoneda \partial^{A}_{1}}
    \arrow[dotted]{d}
&\yoneda A_{0}
    \arrow{r}
    \arrow[dotted]{d}
& \Cok(\yoneda \partial^{A}_{1})
    \arrow{r}
    \arrow{d}{\cong}
&[-0.4cm]0
\\
0
    \arrow{r}
&[-0.4cm]\yoneda B_{d+1}
    \arrow{r}{\yoneda \partial^{B}_{d+1}}
&[0.2cm]\yoneda B_{d}
    \arrow{r}{\yoneda \partial^{B}_{d}}
&\cdots
    \arrow{r}{\yoneda \partial^{B}_{2}}
&\yoneda B_{1}
    \arrow{r}{\yoneda \partial^{B}_{1}}
&\yoneda B_{0}
    \arrow{r}
& \Cok(\yoneda \partial^{B}_{1})
    \arrow{r}
&[-0.4cm]0
\end{tikzcd}
\]
Repeated use of Schanuel's Lemma (see  \cite[Cor.~I.6.4]{Bass-algebraic-k-theory-book}) yields 
\[
\left(\bigoplus_{0 \leq 2i \leq d+1}\yoneda A_{2i}\right) \oplus \left(\bigoplus_{0 \leq 2i+1 \leq d+1}\yoneda B_{2i+1}\right) 
    \cong
    \left(\bigoplus_{0 \leq 2i+1 \leq d+1}\yoneda A_{2i+1}\right) \oplus \left(\bigoplus_{0 \leq 2i \leq d+1}\yoneda B_{2i}\right)
\]
in $\rmod{\cat{C}}$. 
Since $\yoneda$ is fully faithful and additive, this implies 
\[
\left(\bigoplus_{0 \leq 2i \leq d+1} A_{2i}\right) \oplus \left(\bigoplus_{0 \leq 2i+1 \leq d+1} B_{2i+1}\right) 
    \cong
    \left(\bigoplus_{0 \leq 2i+1 \leq d+1} A_{2i+1}\right) \oplus \left(\bigoplus_{0 \leq 2i \leq d+1} B_{2i}\right)
\]
in $\cat{C}$, and hence in $K_{0}^{\sp}(\cat{C})$ we have 
\[
\left(\sum_{0 \leq 2i \leq d+1} [A_{2i}]^{\sp}\right) + \left(\sum_{0 \leq 2i+1 \leq d+1} [B_{2i+1}]^{\sp}\right) 
    =
    \left(\sum_{0 \leq 2i+1 \leq d+1} [A_{2i+1}]^{\sp}\right) + \left(\sum_{0 \leq 2i \leq d+1} [B_{2i}]^{\sp}\right).
\]
Rearranging this gives the desired equation.
\end{proof}

The final ingredient for our first main result is as follows.

\begin{lem}\label{lem:ses-in-modC}
Suppose 
$\begin{tikzcd}[column sep=0.5cm]
0 
    \arrow{r}{} 
& \fun{M'} 
    \arrow{r}{\alpha} 
& \fun{M} 
    \arrow{r}{\beta} 
& \fun{M''}
    \arrow{r}{} 
& 0
\end{tikzcd}$ 
is a short exact sequence in $\rmod{\cat{C}}$. 
Then there are left $d$-exact sequences 
\eqref{eqn:left-d-exact-A} and \eqref{eqn:left-d-exact-B} in $\cat{C}$, satisfying 
$\Cok(\yoneda \partial^{A}_{1}) \cong \fun{M'}$
and 
$\Cok(\yoneda \partial^{B}_{1}) \cong \fun{M''}$, 
such that there is also a left $d$-exact sequence of the form 
\begin{equation}\label{eqn:left-d-exact-A-plus-B}
\begin{tikzcd}
A_{d+1} \oplus B_{d+1}
    \arrow[tail]{r}
&A_{d}\oplus B_{d}
    \arrow{r}
&\cdots
    \arrow{r}
&A_{1} \oplus B_{1}
    \arrow{r}{\partial}
&A_{0}\oplus B_{0} 
\end{tikzcd}
\end{equation}
with $\Cok(\yoneda \partial) \cong \fun{M}$. 
\end{lem}

\begin{proof}
\cref{lem:projective-resolution} ensures the existence of left $d$-exact sequences \eqref{eqn:left-d-exact-A} and \eqref{eqn:left-d-exact-B} so that
\begin{align}
\begin{tikzcd}[column sep=0.9cm,ampersand replacement=\&]
0
    \arrow{r}
\&\yoneda A_{d+1}
    \arrow{r}{\yoneda \partial^{A}_{d+1}}
\&\yoneda A_{d}
    \arrow{r}{\yoneda \partial^{A}_{d}}
\&\cdots
    \arrow{r}{\yoneda \partial^{A}_{2}}
\&\yoneda A_{1}
    \arrow{r}{\yoneda \partial^{A}_{1}}
\&\yoneda A_{0}
    \arrow{r}
\& \fun{M'}
    \arrow{r}
\&0
\end{tikzcd}\label{eqn:proj-res-for-M'}
\\
\begin{tikzcd}[column sep=0.9cm,ampersand replacement=\&]
0
    \arrow{r}
\&\yoneda B_{d+1}
    \arrow{r}{\yoneda \partial^{B}_{d+1}}
\&\yoneda B_{d}
    \arrow{r}{\yoneda \partial^{B}_{d}}
\&\cdots
    \arrow{r}{\yoneda \partial^{B}_{2}}
\&\yoneda B_{1}
    \arrow{r}{\yoneda \partial^{B}_{1}}
\&\yoneda B_{0}
    \arrow{r}
\& \fun{M''}
    \arrow{r}
\&0
\end{tikzcd}\label{eqn:proj-res-for-M''}
\end{align}
are exact in $\rmod{\cat{C}}$. 
Since the objects $\yoneda A_{i}$ and $\yoneda B_{i}$ are projective in $\rmod{\cat{C}}$, we can use the Horseshoe Lemma (see \cite[Lem.~IX.7.8]{Aluffi-Chapter0}) to obtain the  commutative diagram
\begin{equation}\label{eqn:horseshoe}
\begin{tikzcd}[column sep=0.6cm,ampersand replacement=\&]
{}\& 0\arrow{d} \& {}\&0 \arrow{d}\& 0 \arrow{d}\& 0 \arrow{d}\& {} \\
0
    \arrow{r}
\&\yoneda A_{d+1}
    \arrow{r}{\yoneda \partial^{A}_{d+1}}
    \arrow{d}
\&\cdots
    \arrow{r}{\yoneda \partial^{A}_{2}}
\&\yoneda A_{1}
    \arrow{r}{\yoneda \partial^{A}_{1}}
    \arrow{d}
\&\yoneda A_{0}
    \arrow{r}
    \arrow{d}
\& \fun{M'}
    \arrow{r}
    \arrow{d}{\alpha}
\&0
\\
0
    \arrow{r}
\&\yoneda A_{d+1} \oplus \yoneda B_{d+1}
    \arrow{r}{}
    \arrow{d}
\&\cdots
    \arrow{r}{}
\&\yoneda A_{1} \oplus \yoneda B_{1}
    \arrow{r}{\yoneda \partial}
    \arrow{d}
\&\yoneda A_{0} \oplus \yoneda B_{0}
    \arrow{r}
    \arrow{d}
\& \fun{M}
    \arrow{r}
    \arrow{d}{\beta}
\&0
\\
0
    \arrow{r}
\&\yoneda B_{d+1}
    \arrow{r}{\yoneda \partial^{B}_{d+1}}
    \arrow{d}
\&\cdots
    \arrow{r}{\yoneda \partial^{B}_{2}}
\&\yoneda B_{1}
    \arrow{r}{\yoneda \partial^{B}_{1}}
    \arrow{d}
\&\yoneda B_{0}
    \arrow{r}
    \arrow{d}
\& \fun{M''}
    \arrow{r}
    \arrow{d}
\&0\\
{}\& 0 \& {}\&0 \& 0 \& 0 \& {}
\end{tikzcd}
\end{equation}
in $\rmod{\cat{C}}$ with exact rows and columns. 
The functor $\yoneda$ being fully faithful and additive yields a sequence of the form \eqref{eqn:left-d-exact-A-plus-B} in $\cat{C}$, and the exactness of the middle row of \eqref{eqn:horseshoe} says that \eqref{eqn:left-d-exact-A-plus-B} is left $d$-exact with $\Cok(\yoneda \partial) \cong \fun{M}$, as required.
\end{proof}

\begin{thm}\label{thm:thetaC}
There is a group homomorphism $\theta_{\cat{C}}\colon K_{0}(\rmod{\cat{C}})\to K_{0}^{\sp}(\cat{C})$ defined by 
$
\theta_{\cat{C}}([\fun{M}]) 
    = \sum_{i=0}^{d+1}(-1)^{i}[A_{i}]^{\sp}
$, 
 for any left $d$-exact sequence \eqref{eqn:left-d-exact-A} in $\cat{C}$ 
satisfying $\Cok(\yoneda \partial^{A}_{1}) \cong \fun{M}$.
\end{thm}

\begin{proof}
Combining Lemmas~\ref{lem:projective-resolution} and \ref{lem:isoclasses}, we see that $\theta_{\cat{C}}$, as defined in the statement of the theorem, gives a well-defined group homomorphism from the free abelian group generated by $\Iso{(\rmod{\cat{C}})}$ to $K_{0}^{\sp}(\cat{C})$. 
\cref{lem:ses-in-modC} shows that this induces the group homomorphism $\theta_{\cat{C}}\colon K_{0}(\rmod{\cat{C}})\to K_{0}^{\sp}(\cat{C})$ as claimed.
\end{proof}


\subsection{The proof of \texorpdfstring{\cref{theoremA}}{Theorem A}}\label{sec:defining-thetaX}

Since we assume $(\cat{C},\adj{E})$ has $d$-kernels, it follows that $\cat{C}$ has weak kernels. Thus, \cref{setup:section2} is satisfied and the results from \cref{background} apply. 
We are now ready to prove \cref{theoremA}. 
Recall that the $d$-exact subcategory $(\cat{C},\adj{E}_{\cat{X}})\sse(\cat{C},\adj{E})$ was defined in \cref{prop:relative-d-exact-subcategory}, and that 
$\pi_{\cat{X}}([C]^{\sp}) = [C]_{\cat{X}}$ for $C\in\cat{C}$ (see \eqref{eqn:pi_X}).

\begin{proof}[Proof of \cref{theoremA}]
For the existence of $\theta_{\cat{X}}\colon K_{0}(\rmod{\cat{X}})\to K_{0}(\cat{C},\adj{E}_{\cat{X}})$, 
we will show 
$\pi_{\cat{X}}\theta_{\cat{C}}K_{0}(\iota) = 0$, where $\iota$ is the inclusion of $\Ker \restr{(-)}{\cat{X}}$ into $\rmod{\cat{C}}$,
and then use \cref{lemmas_FJS}. 
To this end, let $\fun{M}\in\Ker\restr{(-)}{\cat{X}}$ be arbitrary and, as in \cref{lem:projective-resolution}, let 
\begin{equation*}
\begin{tikzcd}
A_{d+1}
    \arrow[tail]{r}{\partial^{A}_{d+1}}
&A_{d}
    \arrow{r}{\partial^{A}_{d}}
&\cdots
    \arrow{r}{\partial^{A}_{2}}
&A_{1}
    \arrow{r}{\partial^{A}_{1}}
&A_{0}
\end{tikzcd}
\end{equation*}
be a left $d$-exact sequence in $\cat{C}$ with $\Cok(\yoneda \partial^{A}_{1}) \cong \fun{M}$.
We claim that we have 
\[
\pi_{\cat{X}}\theta_{\cat{C}}K_{0}(\iota)([\fun{M}]) 
    = \pi_{\cat{X}}\theta_{\cat{C}}([\fun{M}])
    = \pi_{\cat{X}}\left( \sum_{i=0}^{d+1}(-1)^{i}[A_{i}]^{\sp} \right)
    = \sum_{i=0}^{d+1}(-1)^{i}[A_{i}]_{\cat{X}}
    = 0
\]
in $K_{0}(\cat{C},\adj{E}_{\cat{X}})$, using \cref{thm:thetaC} for the second equality.
Indeed, $\Cok(\yoneda \partial^{A}_{1}) \cong \fun{M}$ implies 
$
\Cok\left(\restr{(\yoneda \partial^{A}_{1})}{\cat{X}}\right)
    = \restr{(\Cok(\yoneda \partial^{A}_{1}))}{\cat{X}} 
    \cong \restr{\fun{M}}{\cat{X}} 
    = 0
$ 
since $\restr{(-)}{\cat{X}}$ is an exact functor and $\fun{M}\in\Ker\restr{(-)}{\cat{X}}$. 
But this means that 
$\restr{(\yoneda \partial^{A}_{1})}{\cat{X}}$ 
is an epimorphism, and hence \eqref{eqn:left-d-exact-A} is an $\adj{E}_{\cat{X}}$-admissible $d$-exact sequence (see \cref{prop:relative-d-exact-subcategory}). 
Thus, $\sum_{i=0}^{d+1}(-1)^{i}[A_{i}]_{\cat{X}}$ vanishes in $K_{0}(\cat{C},\adj{E}_{\cat{X}})$ (see \cref{def:grothendieck-group}), as claimed.

Hence, \cref{lemmas_FJS} yields a unique group homomorphism $\theta_{\cat{X}}\colon K_{0}(\rmod{\cat{X}})\to K_{0}(\cat{C},\adj{E}_{\cat{X}})$
such that the following diagram commutes. 
\begin{equation}\label{eqn:thetaC-thetaX}
\begin{tikzcd}[column sep=1.7cm]
K_{0}(\rmod{\cat{C}}) 
    \arrow{r}{K_0(\restr{(-)}{\cat{X}})}
    \arrow{d}[swap]{\pi_{\cat{X}}\theta_{\cat{C}}}
& K_{0}(\rmod{\cat{X}})
    \arrow[dotted]{dl}{\theta_{\cat{X}}}
\\
K_{0}(\cat{C},\adj{E}_{\cat{X}})
&{}
\end{tikzcd}
\end{equation}
It follows from \cref{thm:thetaC}, the exactness of $\restr{(-)}{\cat{X}}$ and the commutativity of \eqref{eqn:thetaC-thetaX} that $\theta_{\cat{X}}$ satisfies 
$\theta_{\cat{X}}([\Cok\left(\restr{(\yoneda \partial^{A}_{1})}{\cat{X}}\right)])
    = \sum_{i=0}^{d+1}(-1)^{i}[A_{i}]_{\cat{X}}$
whenever \eqref{eqn:left-d-exact-A} is an $\adj{E}$-admissible $d$-exact sequence, concluding the proof.
\end{proof}

\begin{proof}[Proof of Proposition~\ref{prop_uniqueness}.]
We first show that $\theta_{\cat{X}}$ respects the stated property. Let
\begin{equation}\label{left_dexact}
\begin{tikzcd}
A_{d+1}
    \arrow[tail]{r}{\partial^{A}_{d+1}}
&A_{d}
    \arrow{r}{\partial^{A}_{d}}
&\cdots
    \arrow{r}{\partial^{A}_{2}}
&A_{1}
    \arrow{r}{\partial^{A}_{1}}
&A_{0}
\end{tikzcd}
\end{equation}
be a left $d$-exact sequence in $\cat{C}$. Then
\begin{equation*}
\begin{tikzcd}
0
    \arrow{r}
&\yoneda A_{d+1}
    \arrow{r}{\yoneda \partial^{A}_{d+1}}
&\cdots
    \arrow{r}{\yoneda \partial^{A}_{2}}
&\yoneda A_{1}
    \arrow{r}{\yoneda \partial^{A}_{1}}
&\yoneda A_{0}
    \arrow{r}
& \Cok\yoneda\partial^A_1
    \arrow{r}
&0
\end{tikzcd}
\end{equation*}
is exact in $\rmod{\cat{C}}$. Since $\restr{(-)}{\cat{X}}$ is exact, letting $\yoneda_\cat{X}:= \restr{\yoneda(-)}{\cat{X}}$, we have that
\begin{equation*}
\begin{tikzcd}
0
    \arrow{r}
&\yoneda_\cat{X} A_{d+1}
    \arrow{r}{\yoneda_\cat{X} \partial^{A}_{d+1}}
&\cdots
    \arrow{r}{\yoneda_\cat{X} \partial^{A}_{2}}
&\yoneda_\cat{X} A_{1}
    \arrow{r}{\yoneda_\cat{X} \partial^{A}_{1}}
&\yoneda_\cat{X} A_{0}
    \arrow{r}
& \Cok\yoneda_\cat{X}\partial^A_1
    \arrow{r}
&0
\end{tikzcd}
\end{equation*}
is exact in $\rmod{\cat{X}}$ and so $[\Cok (\restr{\cat{C}(-,\partial_1^A)}{\cat{X}})]_\cat{X}=[\Cok\yoneda_\cat{X}\partial^A_1]_\cat{X}=\sum_{i=0}^{d+1}(-1)^i[\yoneda_\cat{X}A_i]_\cat{X}$. Hence
\begin{align*}
    \theta_\cat{X}[\Cok (\restr{\cat{C}(-,\partial_1^A)}{\cat{X}})]_\cat{X}
    &=
    \sum_{i=0}^{d+1}(-1)^i\theta_\cat{X}[\yoneda_\cat{X}A_i]_\cat{X}
    =\sum_{i=0}^{d+1}(-1)^i\theta_\cat{X}K_0(\restr{(-)}{\cat{X}})[\yoneda A_i]\\
    &=\sum_{i=0}^{d+1}(-1)^i\pi_\cat{X}\theta_\cat{C}[\yoneda A_i]
    =\sum_{i=0}^{d+1}(-1)^i\pi_\cat{X}[A_i]^{\sp}\\
    &=\sum_{i=0}^{d+1}(-1)^i[A_i]_\cat{X}
\end{align*}
as wished.

We now prove that $\theta_\cat{X}$ is unique with respect to this property. Suppose that
$F: K_0(\rmod{\cat{X}})\to K_0(\cat{C},\cat{E}_\cat{X})$
is a group homomorphism such that if (\ref{left_dexact}) is a left $d$-exact sequence in $\cat{C}$, then $F[\Cok (\restr{\cat{C}(-,\partial_1^A)}{\cat{X}})]_\cat{X}=\sum_{i=0}^{d+1}(-1)^i[A_i]_\cat{X}$.
By the universal property of cokernels, to prove that $F=\theta_\cat{X}$, it is enough to show that $F\circ K_0(\restr{(-)}{\cat{X}})=\pi_\cat{X}\circ\theta_\cat{C}$ on generators $[\fun{M}]\in K_0(\rmod{\cat{C}})$.
By Lemma~\ref{lem:projective-resolution}, for any $\fun{M}\in\rmod{\cat{C}}$, there is a left $d$-exact sequence in $\cat{C}$ of the form (\ref{left_dexact}) such that
\begin{equation*}
\begin{tikzcd}
0
    \arrow{r}
&\yoneda A_{d+1}
    \arrow{r}{\yoneda \partial^{A}_{d+1}}
&\cdots
    \arrow{r}{\yoneda \partial^{A}_{2}}
&\yoneda A_{1}
    \arrow{r}{\yoneda \partial^{A}_{1}}
&\yoneda A_{0}
    \arrow{r}
&\fun{M}
    \arrow{r}
&0
\end{tikzcd}
\end{equation*}
is exact in $\rmod{\cat{C}}$. Applying the exact functor $\restr{(-)}{\cat{X}}$ to it, we obtain an exact sequence in $\rmod{\cat{X}}$ and so
\[
[\Cok\yoneda_\cat{X}{\partial_1^A}]_\cat{X}=[\restr{\fun{M}}{\cat{X}}]_\cat{X}=\sum_{i=0}^{d+1}(-1)^i[\yoneda_\cat{X}{A_i}]_\cat{X}
\]
in $K_0(\rmod{\cat{X}})$.
By the assumption on $F$, we then have that
\begin{align*}
    F\circ K_0(\restr{(-)}{\cat{X}})[\fun{M}]
    &=F[\restr{\fun{M}}{\cat{X}}]_\cat{X}
    =\sum_{i=0}^{d+1}(-1)^i[A_i]_\cat{X}
    =\pi_\cat{X}\theta_\cat{C}\Big(\sum_{i=0}^{d+1}(-1)^i[\yoneda{A_i}]\Big)
    =\pi_\cat{X}\theta_\cat{C}[\fun{M}],
\end{align*}
concluding the proof.
\end{proof}


\appendix 
\section{Examples of \texorpdfstring{$d$}{d}-exact categories with \texorpdfstring{$d$}{d}-kernels}\label{appendix}

Recall that in \cref{setup:section4} for our main results, we assume that $(\cat{C},\adj{E})$ is a skeletally small, idempotent complete, $d$-exact category that has $d$-kernels. 
We motivate our choice in working in this generality in the present appendix. 
In classical homological algebra, i.e.\ when $d=1$, this reduces to exhibiting meaningful examples of exact categories with kernels. To start with, abelian categories, which are exact and have kernels, are in abundance in homological algebra. Let us explore some other examples.

\begin{example}[Torsion(-free) classes]\label{example:torsion-class}
Let $\cat{U}$ be a torsion class in an abelian category $\cat{C}$. Then $\cat{U}$ is extension-closed in $\cat{C}$, and hence inherits an exact structure from the abelian structure on $\cat{C}$. Furthermore, $\cat{U}$ has kernels and cokernels; see e.g.\ \cite[Sec.~5.4]{BondalvandenBergh2003}. 
Note that, unlike the cokernel, the kernel of a morphism $f$ in $\cat{U}$ will not in general agree with the kernel of $f$ in $\cat{C}$. 
Dual statements hold for a torsion-free class in $\cat{C}$.
\end{example}

By \cite[p.~193, Cor.]{Rump-almost-abelian-cats} (see also 
\cite[Prop.~B.3]{BondalvandenBergh2003}), we can view torsion(-free) classes through the lens of quasi-abelian categories.

\begin{example}[Quasi-abelian categories]\label{example:quasi-abelian}
A \emph{quasi-abelian} category is an additive category that has kernels and cokernels, and in which the pushout (resp.\ pullback) of each kernel (resp.\ cokernel) is again a kernel (resp.\ cokernel) (see e.g.\ \cite[Def.~2.3]{HSW21}). 
The class of all kernel-cokernel pairs in a quasi-abelian category $\cat{C}$ forms an exact structure on $\cat{C}$ \cite[Rem.~1.1.11]{Schneiders-quasi-abelian-categories-and-sheaves}, and hence $\cat{C}$ is an exact category with kernels (and cokernels). 
Note that this exact structure on quasi-abelian category is intrinsic to the category and usually not the split exact structure.

Besides algebraic settings, quasi-abelian categories show up in more analytic fields, such as 
functional analysis and harmonic analysis. Indeed, the category of Banach spaces \cite[Prop.~3.1.7]{Prosmans-algebre-homologique-quasi-abelienne} and the category of topological abelian groups \cite[Sec.~2]{Rump-almost-abelian-cats} are quasi-abelian, amongst many other examples. 
\end{example}

There is a wider class of examples that Examples~\ref{example:torsion-class} and \ref{example:quasi-abelian} belong to.

\begin{example}(Pre-abelian categories)
An additive category $\cat{C}$ is said to be \emph{pre-abelian} if every morphism in $\cat{C}$ has a kernel and a cokernel in $\cat{C}$. 
By \cite[Cor.~2]{Rump-on-the-maximal-exact-structure-on-an-additive-category} (see also \cite[Thm.~3.5]{Crivei2012} and \cite[Thm.~3.3]{SiegWegner-maximal-exact-structures-on-additive-categories}), one can equip $\cat{C}$ with a (unique) maximal exact structure, yielding an exact category that has kernels. However, it is harder to say explicitly what the admissible conflations are for an arbitrary pre-abelian category, unlike for a quasi-abelian one.

Examples of pre-abelian categories that are not quasi-abelian include the categories of 
complete Hausdorff locally convex spaces and of 
bornological (Hausdorff and non-Hausdorff) locally convex spaces
over the real (or complex) numbers. 
See \cite[Fig.~1]{HSW21} for a recent overview of various pre-abelian categories and their properties.
\end{example}

The categories in the examples above have \emph{both} kernels and cokernels. However, we emphasise that there are also examples of exact categories with kernels that typically do not have cokernels. 
We give two such examples below. 
\cref{exam:modC} is a category-theoretic and \cref{exam:new-projR} is ring-theoretic.

\begin{example}
\label{exam:modC}
For an additive category $\cat{C}$, the category $\rmod{\cat{C}}$ always has cokernels, but it has kernels if and only if $\cat{C}$ has weak kernels; see \cite[p.~41, Prop.]{Auslander-representation-dimension-of-artin-algebras-QMC}. Therefore, $(\rmod{\cat{C}})^{\op}$ always has kernels, but not necessarily cokernels. 
\end{example}

The authors are very grateful to Raphael Bennett-Tennenhaus for several discussions in preparation of \cref{exam:new-projR}, and to Dag Oskar Madsen for communicating the ring $D$ we define below to them.

\begin{example}
\label{exam:new-projR}
In this example, we will give a triangular matrix ring $D$ that is: 
\begin{itemize}
    \item right artinian, right noetherian and right coherent; 
    \item semiprimary with (right and left) global dimension 2; and 
    \item not left coherent.
\end{itemize}
From this we can deduce that the category of finitely presented (equivalently, finitely generated) right $D$-modules $\rproj{D}$ has kernels but does not have arbitrary cokernels, using the following facts.

Let $R$ be a ring. 
We equip the idempotent complete, 
additive category $\rproj{R}$ with its split exact structure. 
By \cite[Exam.~4.2, Prop.~4.5(1), and the second paragraph on p.~158]{Beligiannis-on-the-freyd-cats-of-additive-cats}, we have that $\rmod{R}$ is abelian if and only if $R$ is right coherent. 
Any right noetherian ring is right coherent by \cite[Exam.~(4.46)(a)]{Lam99}.

In addition, 
we know $\rmod{R}$ has enough projectives by \cite[Prop.~3.6(1) and Exam.~4.2]{Beligiannis-on-the-freyd-cats-of-additive-cats} and $\rproj{R}$ is subcategory of projectives in $\rmod{R}$. 
Thus, if $M\in\rmod{R}$, then we define $\rpd{\rmod{R}}{M}$ to be the minimal length (possibly infinite) of a resolution of $M$ by objects in $\rproj{R}$. 
There are obvious left versions of these definitions too. The next lemma then follows from Schanuel's lemma; see \cite[Ch.~7, 1.2]{McConnellRobson}.

\begin{lem}\label{lem:pd-is-rpd-for-noetherian-rings}
Suppose $R$ is a right noetherian ring and suppose $M\in\rmod{R}$ is a finitely presented (equivalently, finitely generated) right $R$-module.
Then the (usual) projective dimension $\pd{\rMod{R}}{M}$ of $M$ is equal to $\rpd{\rmod{R}}{M}$.  
\end{lem}



Since there is an the equivalence 
$\rproj{R} \simeq (\rproj{R^{\op}})^{\op} = (\lproj{R})^{\op}$ 
(see e.g.\ \cite[Prop.~2.3]{Krause-KRS}), 
it follows from \cite[Exam.~4.2, and Prop.~4.5(1), (6)]{Beligiannis-on-the-freyd-cats-of-additive-cats} that 
$\rproj{R}$ has kernels 
	(resp.\ cokernels) 
if and only if 
$R$ is right coherent and $\rpd{\rmod{R}}{M}\leq 2$ for all $M\in\rmod{R}$
	(resp.\ $R$ is left coherent and $\lpd{\lmod{R}}{M}\leq 2$ for all $M\in\lmod{R}$).
In particular, it follows from \cref{lem:pd-is-rpd-for-noetherian-rings} that if $R$ is right noetherian and $\rgldim{R} \leq 2$, then $\rproj{R}$ has kernels.

Now we shall define the ring $D$, which combines the rings from \cite[Exam.~(4.66)(e)]{Lam99} and \cite{Madsen}. 
Let $K\deff \BQ(x_{1},x_{2},\ldots)$ be the field of rational functions over $\BQ$ in a countably infinite number of indeterminates. 
Define a field monomorphism $\phi \colon K \to K$ by $\phi(x_{i}) \deff x_{i+1}$ for all $i \geq 1$. We define a $(K,K)$-bimodule $N$ as follows: 
the underlying set of $N$ is $K$, 
the right $K$-action on $N$ is the usual action (i.e.\ $N_{K}$ is the regular module $K_{K}$), and 
left $K$-action on $N$ is 
$\lambda \cdot n = \phi(\lambda)n$ 
for all $\lambda\in K$ and $n\in N$. 
Notice that $N$ is a $1$-dimensional right $K$-module, but it is an infinite-dimensional left $K$-module with basis $\set{1, x_{1}, x_{1}^{2}, x_{1}^{3}, \ldots}$. 
Finally, we put 
\[
T \deff 
    \begin{pmatrix}
        K & N & N \\
        0 & K & K \\
        0 & 0 & K
    \end{pmatrix},
\hspace{0.5cm}
L \deff 
    \begin{pmatrix}
        0 & 0 & N \\
        0 & 0 & 0 \\
        0 & 0 & 0
    \end{pmatrix} 
\hspace{0.5cm}
\text{and}
\hspace{0.5cm}
D \deff T/L. 
\]

By \cite[Thm.~(1.22)]{Lam-FC}, we see that $D$ is right artinian. Hence, $D$ is right noetherian and semiprimary by \cite[Thm.~(4.15)]{Lam-FC}, and it is right coherent by \cite[Exam.~(4.46)(a)]{Lam99}. 
Since $D$ is semiprimary, we have 
$\rgldim{D} = \lgldim{D}$ by \cite[Cor.~9]{Auslander55}. 
Arguing as in \cite[p.~57, (6)]{Lam-FC}, we have that the Jacobson radical of $D$ is 
\[
\rad D = 
\left.\begin{pmatrix}
    0 & N & N \\
    0 & 0 & K \\
    0 & 0 & 0
\end{pmatrix}
\middle/ 
L
\right.
\] 
and hence the simple right $D$-modules are 
$
\begin{pmatrix}
    K & 0 & 0
\end{pmatrix},$ 
$
\begin{pmatrix}
    0 & K & 0
\end{pmatrix}$ 
and 
$\begin{pmatrix}
    0 & 0 & K
\end{pmatrix}$, which have projective dimensions $2$, $1$ and $0$, respectively. 
Thus, $\rgldim{D} = 2$ by \cite[Cor.~11]{Auslander55}. 
It follows that $\rproj{D}$ has kernels.

Lastly, to see that $D$ is not left coherent 
consider the left ideal $I/L$ of $D$, where 
\[
I \deff 
    \begin{pmatrix}
        0 & 0 & N \\
        0 & 0 & K \\
        0 & 0 & 0
    \end{pmatrix}.
\]
There is a left $D$-module epimorphism 
$\alpha \colon D \onto I/L$
given by
\[
\alpha 
    \left(
    \begin{pmatrix}
        a & b & c \\
        0 & d & e \\
        0 & 0 & f
    \end{pmatrix}
    + L
    \right)
    \deff 
    \begin{pmatrix}
        a & b & c \\
        0 & d & e \\
        0 & 0 & f
    \end{pmatrix}
    \cdot
    \begin{pmatrix}
        0 & 0 & 0 \\
        0 & 0 & 1 \\
        0 & 0 & 0
    \end{pmatrix}
    + L
    =
    \begin{pmatrix}
        0 & 0 & 0 \\
        0 & 0 & d \\
        0 & 0 & 0
    \end{pmatrix}
    + L,
\]
which has kernel
\[
\Ker\alpha 
    =
    \left.
    \begin{pmatrix}
        K & N & N \\
        0 & 0 & K \\
        0 & 0 & K
    \end{pmatrix}
    \middle/ L
    \right..
\]
Moreover, notice that $\Ker\alpha$ is not a finitely generated left $D$-module because $N$ is not finitely generated as a left $K$-module. 
Thus, the short exact sequence 
\[
\begin{tikzcd}
0
    \arrow{r}
& \Ker\alpha
    \arrow{r}
& D
    \arrow{r}{\alpha}
& I/L
    \arrow{r}
& 0
\end{tikzcd}
\]
demonstrates that $I/L$ is a finitely generated left ideal of $D$ that cannot be finitely presented by \cite[Prop.~(4.26)(b)]{Lam99}. In particular, $D$ is not left coherent,  
which concludes this example.
\end{example}

We now turn our attention to higher homological algebra and prove a $d$-analogue of \cref{example:torsion-class} where $d\geq 1$ is an integer. That is, we will show that a $d$-torsion class of a suitable $d$-abelian category is $d$-exact (see \cref{def:d-exact-category}) and has $d$-kernels (see \cref{def:has-d-kernels}). 
We recall below only the definitions we will use explicitly.

A \emph{$d$-abelian} category is defined using the notions of $d$-(co)kernels and $d$-exact sequences (see Definitions~\ref{def:d-kernels} and \ref{def:d-exact-sequence}); see \cite[Def.~3.1]{Jasso-n-abelian-and-n-exact-categories}. In particular, each morphism in a $d$-abelian category has a $d$-kernel. 
Combining \cite[Thm.~3.16]{Jasso-n-abelian-and-n-exact-categories} and \cite[Cor.~1.3(ii)]{Kvamme2022}, we know that a skeletally small additive category is $d$-abelian if and only if it is $d$-cluster tilting in an abelian category.

\begin{defn}\label{def:d-torsion}
\cite[Def.~1.1]{Jorgensen-torsion-classes-and-t-structures-in-higher-homological-algebra} 
Suppose $\cat{C}$ is $d$-abelian. A full subcategory $\cat{U}\sse\cat{C}$ is a \emph{$d$-torsion class} if, for each $B\in\cat{C}$, there is a $d$-exact sequence
\[\begin{tikzcd}
U_{B}
    \arrow[tail]{r}{u_{d+1}}
&B 
    \arrow{r}{u_{d}}
&V_{d-1} 
    \arrow{r}{u_{d-1}}
&\cdots 
    \arrow[two heads]{r}{u_{1}}
&V_{0} 
\end{tikzcd}
\]
in $\cat{C}$, such that
$U_{B}\in\cat{U}$
and 
the sequence
\[
\begin{tikzcd}[column sep=1.8cm]
0
    \arrow{r}
&[-20pt] \restr{(\yoneda V_{d-1})}{\cat{U}}
    \arrow{r}{\restr{(\yoneda u_{d-1})}{\cat{U}}}
&\restr{(\yoneda V_{d-2})}{\cat{U}}
    \arrow{r}{\restr{(\yoneda u_{d-2})}{\cat{U}}}
&\cdots
    \arrow{r}{\restr{(\yoneda u_{1})}{\cat{U}}}
&\restr{(\yoneda V_{0})}{\cat{U}}
    \arrow{r}
&[-20pt]0
\end{tikzcd}
\]
is exact in $\rMod{\cat{U}}$.
\end{defn}

\begin{rem}\label{rem:monic-right-approx}
In \cref{def:d-torsion}, notice that the morphism 
$u_{d+1}\colon U_{B} \rightarrowtail B$ 
is a monic right $\cat{U}$-ap\-prox\-i\-ma\-tion of $B$ by \cite[Lem.~2.7(i)(b)]{Jorgensen-torsion-classes-and-t-structures-in-higher-homological-algebra}. This fact is fundamental for the next result.
\end{rem}

\begin{prop}\label{prop:d-torsion-has-d-kernels}
Any $d$-torsion class $\cat{U}$ in a $d$-abelian category $\cat{C}$ has $d$-kernels.
\end{prop}

\begin{proof}
    Let $f_{1}\colon U_{1}\to U_{0}$ be an arbitrary morphism in $\cat{U}$. 
    Since $\cat{C}$ is $d$-abelian, we know $f_{1}$ has a $d$-kernel in $\cat{C}$, giving rise to a left $d$-exact sequence
    \[
    \begin{tikzcd}
        B_{d+1}
            \arrow[tail]{r}{f_{d+1}}
        & B^d
            \arrow{r}{f_{d}}
        & \cdots
            \arrow{r}{f_{3}}
        & B_{2}
            \arrow{r}{f_{2}}
        & U_{1}
            \arrow{r}{f_{1}}
        & U_{0}.
    \end{tikzcd}
    \]
    By \cref{rem:monic-right-approx}, we may take 
    a monic right $\cat{U}$-approximation 
    $a_{i}\colon U_{i}\to B_{i}$ of $B_{i}$ for each $2\leq i \leq d+1$. 
    For each $i=3,\ldots,d+1$, we have the morphism 
    $f_{i}a_{i}\colon U_{i} \to B_{i-1}$, and so there exists 
    $g_{i}\colon U_{i} \to U_{i-1}$ 
    such that $a_{i-1}g_{i} = f_{i}a_{i}$. 
    This gives rise to the commutative diagram 
    \begin{equation}\label{eqn:approx-diagram}
    \begin{tikzcd}
        0
            \arrow{r}
        & U_{d+1}
            \arrow[]{r}{g_{d+1}}
            \arrow[tail]{d}{a_{d+1}}
        & U_{d}
            \arrow[]{r}{g_{d}}
            \arrow[tail]{d}{a_{d}}
        & \cdots
            \arrow[]{r}{g_{3}}
        & U_{2}
            \arrow{r}{g_{2}}
            \arrow[tail]{d}{a_{2}}
        & U_{1}
            \arrow{r}{g_{1}}
            \arrow[equals]{d}{a_{1}}
        & U_{0}
            \arrow[equals]{d}
        \\
        0 
            \arrow{r}
        & B_{d+1}
            \arrow[tail]{r}{f_{d+1}}
        & B_{d}
            \arrow{r}{f_{d}}
        & \cdots
            \arrow{r}{f_{3}}
        & B_{2}
            \arrow{r}{f_{2}}
        & U_{1}
            \arrow{r}{f_{1}}
        & U_{0},
    \end{tikzcd}
    \end{equation}
    where $g_{2}\deff f_{2}a_{2}$, $g_{1}\deff f_{1}$ and $a_{1} \deff \id{U_{1}}$. 

    We claim that the top row of \eqref{eqn:approx-diagram} is a left $d$-exact sequence in $\cat{U}$. First note that $g_{d+1}$ is indeed a monomorphism because $a_{d}g_{d+1} = f_{d+1}a_{d+1}$ is monic. Thus, it remains to show that $g_{i+1}$ is a weak kernel of $g_{i}$ for each $1\leq i \leq d$. 
    Fix $i\in\set{1,\ldots, d}$ and suppose $h\colon U \to U_{i}$ is now a morphism in $\cat{U}$ with $g_{i}h = 0$. Then 
    $f_{i}a_{i}h = a_{i-1}g_{i}h = 0$
    implies there exists $b\colon U \to B_{i+1}$ such that $a_{i}h = f_{i+1}b$. Since $a_{i+1} \colon U_{i+1} \to B_{i+1}$ is a right $\cat{U}$-approximation, there exists $c\colon U \to U_{i+1}$ with $b = a_{i+1}c$. In particular, we have 
    $
    a_{i}h 
        = f_{i+1}b 
        = f_{i+1}a_{i+1}c 
        = a_{i}g_{i+1}c
    $, 
    so $h = g_{i+1}c$ as $a_{i}$ is monic, completing the proof.
\end{proof}

\begin{example}[$d$-torsion(-free) classes]\label{example:d-torsion-free-have-d-kernels}
We explain how, under reasonable assumptions, a $d$-torsion class $\cat{U}$ of a skeletally small, Krull-Schmidt, $d$-abelian category $\cat{C}$ meets the conditions required of the category $\cat{C}$ in \cref{setup:section4}. 
It is clear that $\cat{U}$ is skeletally small as $\cat{C}$ is, and $\cat{U}$ is idempotent complete by \cite[Lem.~2.7(iii)]{Jorgensen-torsion-classes-and-t-structures-in-higher-homological-algebra}. 
By \cite[Cor.~1.3(ii)]{Kvamme2022}, the $d$-abelian category $\cat{C}$ embeds in an abelian category $\cat{A}$ as a $d$-cluster tilting subcategory. If $\cat{A}$ is a finite length category, then $\cat{U}$ is closed under $d$-extensions and $d$-quotients by \cite[Prop.~3.11]{AugustHauglandJacobsenKvammePaluTreffinger}. 
(Actually, one need only assume $\cat{A}$ is noetherian or has arbitrary coproducts for then the existence of a smallest torsion class in $\cat{A}$ containing $\cat{U}$ is guaranteed; see \cite[Sec.~1.3]{pavon2023torsion}.) 
It follows from closure under $d$-quotients that $\cat{U}$ has $d$-cokernels and that they agree with $d$-cokernels taken in $\cat{C}$ (see \cite[Def.~3.7, Rem.~3.12]{AugustHauglandJacobsenKvammePaluTreffinger}). 
Furthermore, \cref{prop:d-torsion-has-d-kernels} shows that $\cat{U}$ has $d$-kernels.

It remains to show that $\cat{U}$ is $d$-exact. By \cite[Thm.~4.4]{Jasso-n-abelian-and-n-exact-categories}, we have that $\cat{C}$ is a $d$-exact category, and hence the class $\adj{E}$ of all $d$-exact sequences in $\cat{C}$ with all terms in $\cat{U}$ forms a $d$-exact structure on $\cat{U}$ by \cite[Cor.~4.15]{Klapproth-n-extension-closed}; this is also observed in \cite[Cor.~3.19]{AugustHauglandJacobsenKvammePaluTreffinger}. 
Thus, $(\cat{U},\adj{E})$ is a $d$-exact category with $d$-kernels and $d$-cokernels.

Dual statements hold for a $d$-torsion-free class in $\cat{C}$.
\end{example}

\begin{acknowledgements}
We thank Raphael Bennett-Tennenhaus for his time and several discussions during the development of
\cref{exam:new-projR}, and Dag Oskar Madsen for email communications and communicating the ring $T$ in 
\cref{exam:new-projR} to us. We would also like to thank the anonymous referee for comments on an earlier version of the paper, which in particular led to the development of Proposition 1.3.

The first author is supported by the EPSRC Programme Grant EP/W007509/1. The second and third authors are supported by a DNRF Chair from the Danish National Research Foundation (grant DNRF156), by a Research Project 2 from the Independent Research Fund Denmark (grant 1026-00050B), and by the Aarhus University Research Foundation (grant AUFF-F-2020-7-16).
\end{acknowledgements}

\end{document}